\documentclass[11pt, reqno]{amsart}
\usepackage{amssymb,latexsym,amsmath,amsfonts,mathdots,enumitem}
\usepackage{latexsym}
\usepackage[mathscr]{eucal}
\usepackage{colortbl,xcolor}
\usepackage{lmodern}
\usepackage{sansmathaccent}
\usepackage[latin1]{inputenc}
\usepackage{tikz}
\usetikzlibrary{shapes,arrows}
\usetikzlibrary{matrix,calc,shapes,arrows,positioning}
\pdfmapfile{+sansmathaccent.map}

\numberwithin{equation}{section}
\theoremstyle{plain}
\newtheorem{example}{Example}

\newtheorem{thm}{Theorem}[section]

\newtheorem{cor}[thm]{Corollary}
\newtheorem{lem}[thm]{Lemma}
\newtheorem{prop}[thm]{Proposition}
\theoremstyle{definition}
\newtheorem{defn}[thm]{Definition}
\newtheorem{rem}[thm]{Remark}

\numberwithin{equation}{section}

\def\C{{\mathbb C}}

\def\I{\mathcal{I}}

\def\beq{\begin{eqnarray}}
\def\eeq{\end{eqnarray}}
\def\beqa{\begin{eqnarray*}}
\def\eeqa{\end{eqnarray*}}

\def\T{\boldsymbol T}

\def\beqn{\begin{equation}}
\def\eeqn{\end{equation}}

\def\mg#1{}

\renewcommand{\epsilon}{\varepsilon}
\renewcommand{\phi}{\varphi}

\begin{document}
\title{Necessary conditions for existence of $\Gamma_n$-contractions and examples of $\Gamma_3$-contractions}
\author{Shubhankar Mandal and Avijit Pal}
\address[S. Mandal]{Department of Mathematics , Indian Institute of Technology, Bhilai }
\email{S. Mandal:shubhankarm@iitbhilai.ac.in}

\address[A. Pal]{Department of Mathematics , Indian Institute of Technology, Bhilai }
\email{A. Pal:avijit@iitbhilai.ac.in}



\begin{abstract}
The fundamental result of B. Sz. Nazy states that every contraction has a coisometric extension and a unitary dilation. The isometric dilation of a contraction on a Hilbert space motivated whether this theory can be extended sensibly to families of operators. It is natural to ask whether this idea can be generalized, where the contraction $T$ is substituted by a commuting $n$-tuples of operators $(S_1,\cdots, S_n)$ acting on some Hilbert space having $\Gamma_n$ as a spectral set. We derive the necessary conditions for the existence of a $\Gamma_n$-isometric dilation for $\Gamma_n$-contractions. Also we discuss an example of a $\Gamma_3$-contraction $(S_1, S_2, S_3)$ acting on some Hilbert space $\mathcal H,$ which has a $\Gamma_3$-isometric dilation, but it fails to satisfy the following condition:  $$E_1^*E_1-E_1E_1^*= E_2^*E_2-E_2E_2^*,$$ where $E_1$ and $E_2$ are the fundamental operators of  $(S_1, S_2, S_3),$ $(S_1,S_2)$ is a pair of commuting contractions and $S_3$ is a partial isometry. Thus, the set of sufficient conditions for the existence of a $\Gamma_3$-isometric dilation breaks down, in general, to be necessary, even when the $\Gamma_3$-contraction $(S_1, S_2, S_3)$ has the special structure as  described above. 

\end{abstract}
\maketitle
\vskip-.5cm

\section{{\bf{Introduction}}}
For $n\geq 2,$ let us consider the symmetrization map $\bold s:\mathbb C^n\rightarrow \mathbb C^n$  defined by $$\bold s(\bold z)=(s_1(\bold z),\ldots,s_n(\bold z)),~~ \bold z=(z_1,\ldots,z_n)\in\mathbb C^n,$$ where $s_i(\bold z)=\sum_{1\leq k_1< \ldots< k_i\leq n}z_{k_1}\ldots z_{k_i}$ and $s_0=1.$ The image $\Gamma_n:=\bold s(\bar{\mathbb D}^n)$  is called a closed symmetrized polydisc. We can manifest that the set $\Gamma_n$ is not convex but  it is polynomially convex \cite{stout}. The open symmetrized polydisc is the set $\mathbb G_n:=\bold s(\mathbb D^n)$ and it's distinguished boundary $b\Gamma_n$ can be expressed as $\bold s(\mathbb T^n),$ the image of the $n$-torus $\mathbb T^n$ under the map $\bold s
$ \cite{Zwonek}.

Let $\mathcal O(\Omega)$ denote the algebra of holomorphic functions on some neighbourhood of the compact set $\Omega$ and $\mathcal B(\mathcal H)$ be the set of all bounded operators  acting on some Hilbert space $\mathcal H.$  A compact set $\Omega\subset \mathbb C^m$ is said to be a spectral set for a commuting $m$-tuples of operators $\mathbf{T}=(T_1,\ldots,T_m)$ if $\sigma(\mathbf T)\subseteq  \Omega$ and the homomorphism $\rho_{\mathbf T}:\mathcal O(\Omega)\rightarrow\mathcal B(\mathcal H)$ is contractive. The concepts of spectral set for an operator on an underlying compact subset of $\mathbb C $ was discovered by Von Neumann in a intricate way. The following theorem  says that for any contraction, the closed unit disc is a spectral set.
\begin{thm} \cite[Chapter 1, Corollary 1.2]{paulsen}
	Suppose $T\in \mathcal B(\mathcal H),$ where $\mathcal H$ is a separable complex Hilbert space.  Then
	$$\|p(T)\|\leq \|p\|_{\infty, \mathbb D}:=\sup\{|p(z)|: |z|<1\}$$ if and only if $\|T\|\leq 1.$
\end{thm}
 Dilations in operator theory are a technique of characterizing a given operator as the restriction of well comprehended operator, action on a bigger Hilbert space, to the original space. Hence the operator on a bigger Hilbert space is named as a dilation of the original operator. The subsequent theorem is a slight refined version of Sz.-Nagy dilation theorem \cite[Chapter 1, Theorem 1.1]{paulsen} which says that every contractions has a co-isometric extension and an unitary dilation.
\begin{thm}\cite[Chapter 4 , Theorem
	4.3]{paulsen} Suppose $T\in \mathcal B(\mathcal H).$ Then $T$ has a
	unitary (power) dilation if and only if  there exists a unitary operator $U$ acting on a Hilbert space $\mathcal K \supseteq \mathcal H$
	such that
	$$P_\mathcal H\,p(U)_{|\mathcal H}=p(T),$$ for all polynomials $p.$
\end{thm}
Schaffer had constructed the existence of such unitary dilation for a given contraction $T.$ Obviously, the von Neumann inequality ensued from the existence of a power dilation via the spectral theorem for unitary operators. 
 
 Let $g=\left(\!(g_{ij})\!\right)$  be a matrix valued polynomial on  
 $\Omega,$ we call $\Omega$ as a complete spectral set (complete $\Omega$-contraction) for $\mathbf T$, if $\|g(\mathbf T) \| \leq \|g\|_{\infty,
 	\Omega}$ for all $g\in \mathcal A\otimes \mathcal
 M_k(\mathbb C), k\geq 1$. We will say that a domain $\Omega$ has the property $P$ if the following holds: if $\Omega$ is a spectral set for a commuting $m$-tuples of operators $\mathbf{T},$ then it is a complete spectral set for $\mathbf{T}.$
 A commuting $m$-tuples of operators $\mathbf{T}$ with $\Omega$ as a spectral set, have a $\partial \Omega$ normal  dilation if there exists a Hilbert space $\mathcal K$ containing $\mathcal H$ as a subspace and a commuting $m$-tuples of normal operators $\mathbf{N}=(N_1,\ldots,N_m)$ on $\mathcal K$ with $\sigma(N)\subseteq \partial \Omega$ such that
$P_{\mathcal H}g(\mathbf N)\mid_{\mathcal H}=g(\mathbf T) ~{\rm{for~ all~}} g\in \mathcal O(\Omega).$
In 1969, Arveson \cite{A} showed that a commuting $m$-tuples of operators $\mathbf{T}$ having $\Omega$ as a spectral set for $\mathbf{T},$ admits a  $\partial \Omega$ normal  dilation if and only if it satisfies the property $P.$ 

In $1984,$ J. Agler \cite{agler} showed that annulus has the property $P.$ Recently, M. Dristchell and S. McCullough \cite{michel} proved that for a domain of connectivity $n\geq 2$ does not satisfy  the property $P.$ In the
multi-variable context symmetrized bi-disc and bi-disc due to Agler and Young \cite{young} and Ando \cite{paulsen}  respectively, have the property $P$. The first counter example in the multi-variable contest was given by Parrott\cite{vern} which is for $\mathbb D^n$ for $n > 2.$  As a result, it is natural to ask if contractive linear maps on a (finite dimensional) normed linear space are necessarily completely contractive.  Using Parrott like homomorphisms G.Misra\cite{GM,sastry}, V. Paulsen \cite{vern}, E. Ricard \cite{pisier} proved that if $\Omega$ is  any ball in
  $\mathbb C^m,$ $m\geq 3,$ then $\Omega$ can not have this property $P.$ It was shown in \cite{cv} that if $B_1$ and $B_2$ are not simultaneously diagonalizable via unitary, then $ \Omega_{\mathbf B}:=   \{(z_1
  ,z_2) :\|z_1 B_1 + z_2 B_2 \|_{\rm op} < 1\},$ does not have  property $P,$ where  $\mathbf
  B=(B_1, B_2)$ in $\mathbb C^2 \otimes \mathcal
  M_2(\mathbb C)$ and $B_1$ and $B_2$ are independent.

We will proclaim a commuting $n$-tuple of bounded operators $(S_1,\ldots,S_n)$ as a $\Gamma_n$-contraction if $\Gamma_n$ is the spectral set. 
The closed symmetrized polydisc $\Gamma_{n}$ is called a complete spectral set (complete $\Gamma_{n}-$contraction) for $(S_{1},...,S_{n})$ if $$\|g(S_1,\ldots,S_n) \| \leq \|g\|_{\infty,
	\Gamma_n}~~{\rm{for~~ all}}~~ g\in \mathcal A\otimes \mathcal
M_k(\mathbb C), k\geq 1,$$ where $g=\left(\!(g_{ij})\!\right)$ is a matrix valued polynomial defined on $\Gamma_n$ and $$\|g\|_{\infty,\Gamma_{n}}=\sup\{\|\left(\!(g_{ij}(z))\!\right)\|_{\rm op}: z \in
\Gamma_n\}.$$
Now we  recollect the definitions of $\Gamma_n$-unitary, $\Gamma_n$-isometry and pure $\Gamma_n$-isometry from \cite{SS}.

\small{\begin{defn}
Suppose $(S_1,\ldots,S_n)$ is a commuting $n$-tuple of operators acting on some Hilbert space $\mathcal H.$ We call $(S_1,\ldots,S_n)$ is
\begin{enumerate}
\item a $\Gamma_n$-unitary if the joint spectrum $\sigma(S_1,\ldots,S_n)\subseteq b\Gamma_n$ and $S_1,\ldots,S_n$ are normal operators on $\mathcal H.$

\item a $\Gamma_n$-isometry if there exists a Hilbert space $\mathcal K\supseteq\mathcal H$ and a $\Gamma_n$-unitary  $(\tilde{S}_1,\ldots,\tilde{S}_n)$ acting on $\mathcal K$ such that $\mathcal H$ is a common invariant subspace for $\tilde{S}_1,\ldots,\tilde{S}_n$ and  $S_i=\tilde{S}_i\mid_{\mathcal H}$ for $i=1,\ldots,n.$ 

\item a $\Gamma_n$-co-isometry if $(S_1^*,\ldots,S_n^*)$ is a $\Gamma_n$-isometry.

\item a pure $\Gamma_n$-isometry  if $(S_1,\ldots,S_n)$ is a $\Gamma_n$-isometry and $S_n$ is  a pure isometry.
\end{enumerate}
\end{defn}}

In section $2,$ we prove the necessary conditions for existence of a $\Gamma_n$-isometric dilation of $\Gamma_n$-contractions. In section $3,$ we consider the $\Gamma_3$-contractions with some special structure. We also construct an example of $\Gamma_3$-contraction which does have $\Gamma_3$-isometric dilation, but at the same time it would not satisfy the following condition:
$$E_1^*E_1-E_1E_1^*= E_2^*E_2-E_2E_2^*,$$ where $E_1$ and $E_2$ are the fundamental operators of  $(S_1, S_2, S_3).$ Therefore, the set of sufficient conditions which is described in Theorem \eqref{main dilation} fail, in general, to be necessary. Still, now, we are not able to produce an example of $\Gamma_n$-contraction, which fails to satisfy one of these necessary conditions. However, the existence of rational dilation for a $\Gamma_n$-contraction for $n\geq 3,$ is still an open question.

\section{\bf{$\Gamma_n$-isometric dilation for $\Gamma_n$-contractions: necessary and sufficient conditions:}}
We begin this section with some basic terminologies that we need further. In this way we are going to recollect few definitions of spectrum, spectral radius and numerical radius of an operator. Assume that $T$ is a bounded linear operator acting on a Hilbert space $\mathcal{H}$. Then spectrum of $T$ is denoted by $\sigma(T)$ and has the succeeding identification $$\sigma(T)=\{\lambda \in \mathbb{C} | \,\, T-\lambda I\,\, \text{is}\,\, \text{not} \,\,\text{invertible}\}.$$ Also the numerical radius of a bounded operator $T$ on some Hilbert space $\mathcal{H}$ explicated as $$\omega(A)=\sup\{|\langle Ax,x\rangle|:\| x \|=1\}.$$ An easy calculation boils up an inequality: $r(A)\leq \omega(A)$, where $r(A)= \sup_{ \lambda \in \sigma(A)}|\lambda|$ is popularly known as spectral radius. Let  $(S_1,\ldots,S_n)$ with $\|S_n\|\leq 1,$ be a commuting $n$-tuple of operators and  $D_{S_n}=(I-S_n^*S_n)^{\frac{1}{2}}$ be the defect operator of $S_n$ and $\mathcal D_{S_n}=\overline{\operatorname{Ran}} D_{S_n}$ be the closure of the range of $D_{S_n}.$ We now look at the succeeding  operator  equations
\begin{equation}\label{fundamental equation} S_i-S_{n-i}^*S_n=D_{S_n}E_iD_{S_n}~~{\rm{and}}~~ S_{n-i}-S_i^*S_n=D_{S_n}E_{n-i}D_{S_n}~{\rm~for} ~i=1,\ldots,(n-1).\end{equation}
These equations are called the fundamental equations and the operators $E_i{,}{\rm{s}}$ for $i=1,\ldots,(n-1)$ are said to be the fundamental operators for $(S_1,\ldots,S_n)$. The following theorem tells the existence and uniqueness of the fundamental operators $E_i^{,}{\rm{s}}$ for $\Gamma_n$-contractions of $(S_1,\ldots,S_n).$
\begin{thm}\cite[Theorem $4.4$]{apal}
For $n \geq 2,$ suppose   $(S_1,\ldots,S_n)$ is a $\Gamma_n$-contraction acting on a Hilbert space $\mathcal H.$ Then there are  operators $E_1,\ldots,E_{n-1} \in \mathcal B(\mathcal D_{S_n})$ which are unique and satisfies \eqref{fundamental equation}. Furthermore, $\omega(E_i+E_{n-i}z) \leq  \binom{n-1}{i}+\binom{n-1}{n-i}$ for all  $z\in \mathbb T.$ 
\end{thm}
For $i=1,2,...,(n-1)$ we define $ k(i)= \binom{n-1}{i}+
\binom{n-1}{n-i}.$ and operator pencil $\Phi^{(i)}_{1}$ and $\Phi^{(i)}_{2}$ for a commuting $n-$tuples of bounded operator $(S_1,...,S_n)$ as
\begin{eqnarray*}
	\Phi_{1}^{(i)}\left(\alpha^{i} S_{i}, \alpha^{n-i} S_{n-i}, \alpha^{n} S_{n}\right)&=&k(i)^{2}\left(1-\left. |\alpha^{n} S_{n}\right|^{2}\right)+\left(|\alpha^{i} S_{i}|^{2}-\left.|\alpha^{n-i} S_{n-i}\right|^{2}\right) \\
	&&-k(i) \alpha^{i}\left(S_{i}-|\alpha|^{2n-2i} S^{*}_{n-i} S_{n}\right)-k(i) \bar{\alpha}^{i}\left(S_{i}^{*}-|\alpha|^{2n-2i} S_{n-i} S_{n}^{*}\right)
\end{eqnarray*}
and 
\begin{eqnarray*}
	\Phi_{2}^{(i)}\left(\alpha^{i} S_{i}, \alpha^{n-i} S_{n-i}, \alpha^{n} S_{n}\right)&=&k(i)^{2}\left(1-\left. |\alpha^{n} S_{n}\right|^{2}\right)+\left(|\alpha^{n-i} S_{n-i}|^{2}-\left.|\alpha^{i} S_{i}\right|^{2}\right) \\
	&&-k(i) \alpha^{n-i}\left(S_{n-i}-|\alpha|^{2i} S^{*}_{i} S_{n}\right)-k(i) \bar{\alpha}^{n-i}\left(S_{n-i}^{*}-|\alpha|^{2i} S_{i} S_{n}^{*}\right).
\end{eqnarray*}
These operator pencils are important tool, used for figuring out the structure of $\Gamma_n-$ contraction. The  succeeding theorem will give a new description of $\Gamma_n$-isometry in terms of these operator pencils, which were obtained  in \cite[Theorem 4.12]{SS} and \cite[Theorem $5.2$ ]{apal}. 

\begin{thm}\label{Gamma_n isometry} Suppose $S_1, \ldots,S_{n}$ are commuting operators acting on a Hilbert space $\mathcal H$ and $\gamma_i=\frac{n-i}{n},$ for $i=1,\ldots,n-1.$
	Then the following are equivalent:
	\begin{enumerate}
		\item $(S_1, \ldots, S_n)$ is a $\Gamma_n$-isometry ;
		
		\item $(\gamma_1S_1,\ldots,\gamma_{n-1}S_{n-1})$ is $\Gamma_{n-1}$-contraction, $S_i = S_{n-i}^*S_n$ and  $S_n$ is a isometry;
		
		\item ( Wold-Decomposition ): there exist an orthogonal decomposition $\mathcal H =
		\mathcal H_1 \oplus \mathcal H_2$ into common invariant subspaces of $S_1,\ldots, S_{n}$  such
		that $(S_1\mid \mathcal H_2 , \ldots , S_n\mid \mathcal H_2 )$ is a pure $\Gamma_n$-isometry and $(S_1\mid \mathcal H_1 , \ldots , S_n\mid \mathcal H_1 )$ is a $\Gamma_n$-unitary;
		
		\item $S_n$ is a isometry and $(S_1, \ldots, S_n)$ is a $\Gamma_n$-contraction;
		
		\item $(\gamma_1S_1,\ldots,\gamma_{n-1}S_{n-1})$ is a $\Gamma_{n-1}$-contraction and  $$\Phi^{(i)}_{1}(\beta^{i}S_i,\beta^{n-i}S_{n-i},\beta^{n}S_n)=0~~{\rm and}~~ \Phi^{(i)}_{2}(\beta^{i}S_{i},\beta^{n-i}S_{n-i},\beta^{n}S_n)= 0 ~~{\rm{ for~~ all}}~~ \beta\in \mathbb T~~{\rm{and}}~~i=1,\ldots,(n-1);$$ 
Furthermore, if  $r(S_i)< \binom{n-1}{i}+\binom{n-1}{n-i}$  for $i = 1,\ldots, (n-1),$  then all of the above are equivalent to :
		
		\item  $(\gamma_1S_1,\ldots,\gamma_{n-1}S_{n-1})$ is a $\Gamma_{n-1}$-contraction along with $(k(i)\beta^nS_n-S_{n-i})(k(i)I-\beta^iS_i)^{-1}$ and $(k(i)\beta^{n}S_n-S_{i})(k(i)I-\beta^{n-i}S_{n-i})^{-1}$ are isometries  for $i = 1,\ldots, (n-1)$ with $ k(i)= \binom{n-1}{i}+
		\binom{n-1}{n-i}$ and for all $\beta\in \mathbb T$  .
	\end{enumerate}
\end{thm}
\subsection{Necessary conditions for existence of a $\Gamma_n$-isometry:}
\paragraph{We now define $\Gamma_n$-isometric dilation of $\Gamma_n$-contraction.} 
\begin{defn}
	A commuting $n$-tuple of operators $(V_1,\ldots,V_{n})$ acting on a Hilbert space $\mathcal K \supseteq \mathcal H$ is called a $\Gamma_n$-isometric dilation of a $\Gamma_n$-contraction $(S_1,\ldots,S_n)$ acting on a Hilbert space $\mathcal H$, if it has the following properties:
	\begin{itemize}
		\item  $(V_1,\ldots,V_{n})$ is $\Gamma_n$-isometric;
		
		\item $P_{\mathcal H}V_1^{m_1}\ldots V_{n}^{m_n}\mid_{\mathcal H}=S_1^{m_1}\ldots S_{n}^{m_n},$ for all  $m_1,\ldots,m_{n-1},m_n \in \mathbb N \cup \{0\}.$
	\end{itemize}
\end{defn}
In the same way, we can define $\Gamma_n$-unitary dilation of a $\Gamma_n$-contraction.

\begin{prop}\label{isometric dilation1}
	Let $(V_1,\ldots,V_{n})$  be a $\Gamma_n$-isometric dilation of a $\Gamma_n$-contraction $(S_1,\ldots,S_n)$ acting on a Hilbert space $\mathcal{H}$. Then $(V_1,\ldots,V_{n})$ is minimal $\Gamma_n$-isometric dilation of $(S_1,\ldots,S_n).$
\end{prop}
\begin{proof}
	Since $(V_1,\ldots,V_{n})$  is a $\Gamma_n$-isometric dilation of $(S_1,\ldots,S_n)$ and $$\mathcal K_{0}=\overline{{\rm{span}}}\{V_1^{m_1}\ldots V_{n}^{m_{n}}h:h\in\mathcal H~{\rm{and}}~m_1,\ldots,m_{n-1},m_n \in \mathbb N\cup\{0\}\}.$$ Clearly, $\mathcal K_{0}$ is an invariant subspace of $V_1^{m_1},\ldots, V_{n}^{m_{n}}$ for any non-negative integers $m_1,\ldots,m_{n}.$ For $i=1,\dots n,$ set $V_{1i}=V_i\mid_{\mathcal K_{0}}.$ Then one can easily verify that $$\mathcal K_{0}=\overline{{\rm{span}}}\{V_{11}^{m_1}\ldots V_{1n}^{m_n}h:h\in\mathcal H~{\rm{and}}~m_1,\ldots,m_{n-1},m_n \in \mathbb N\cup\{0\}\},$$ which implies that for all  $m_1,\ldots,m_{n-1},m_n \in \mathbb N \cup \{0\},$ we get
	$$P_{\mathcal H}(V_1^{m_1}\ldots V_{n}^{m_{n}})h=P_{\mathcal H}(V_{11}^{m_1}\ldots V_{1n}^{m_{n}})h~{\rm{for~all~h\in \mathcal H}}.$$
	$(V_{11},\ldots, V_{1n})$ is obtained from the restriction of a $\Gamma_n$-contraction $(V_1,\ldots,V_n)$ to a common invariant subspace $\mathcal K_{0},$ therefore, $(V_{11},\ldots, V_{1n})$ is  a  $\Gamma_n$-contraction. Since $V_{1n}=V_n\mid_{\mathcal K_{0}}$ and $V_n$ is isometry, which implies that $V_{1n}$ is isometry. Hence, by Theorem \ref{Gamma_n isometry}, $(V_{11},\ldots, V_{1n})$ is a  $\Gamma_n$-isometry. Thus, $(V_{11},\ldots, V_{1n})$ is a minimal $\Gamma_n$-isometry of $(S_1,\ldots,S_n).$
\end{proof}
\begin{prop}\label{isometric dilation2}
	A commuting $n$-tuple of operators $(V_1,\ldots,V_{n})$ acting on Hilbert space $\mathcal K$ is minimal $\Gamma_n$-isometric dilation of a $\Gamma_n$-contraction $(S_1,\ldots,S_n)$ if and only if  $(V_1^*,\ldots,V_{n}^*)$ is a $\Gamma_n$-coisometric extension of $(S_1^*,\ldots,S_n^*).$
\end{prop}
\begin{proof}
	Since $(V_{1},...,V_{n})$ is minimal $\Gamma_n-$isometric dilation of a $\Gamma_n$-contraction $(S_1,\ldots,S_n)$ therefore $$\mathcal K={\overline{\rm{span}}}\{V_1^{m_1}\ldots V_{n}^{m_{n}}h:h\in\mathcal H~{\rm{and}}~m_1,\ldots,m_n \in \mathbb N\cup\{0\}\}.$$ Now, we will prove that $S_iP_{\mathcal H}=P_{\mathcal H}V_i$ for $i=1,\ldots,n,$ where $P_{\mathcal H}:\mathcal K\rightarrow \mathcal H$ is orthogonal projection onto $\mathcal H.$ Hence, for $h\in \mathcal H,$ we have
	\begin{align*}
		S_iP_{\mathcal H}(V_1^{m_1}\ldots V_{n}^{m_{n}}h)=S_i(S_1^{m_1}\ldots S_n^{m_n}h)&=(S_1^{m_1}\ldots S_{i}^{m_{i}+1}\ldots S_{n-1}^{m_{n-1}}S_n^{m_n}h)\\&=P_{\mathcal H}(V_1^{m_1}\ldots V_{i}^{m_{i}+1}\ldots V_{n}^{m_{n}}h)\\&=P_{\mathcal H}V_i(V_1^{m_1}\ldots V_{n}^{m_{n}}h)
	\end{align*}
	which gives  $S_iP_{\mathcal H}=P_{\mathcal H}V_i$ for $i=1,\ldots,n.$ 
	Also, for any $h \in \mathcal H$ and  $k \in \mathcal K,$ we obtain
	$$\langle S_i^*h,k\rangle=\langle P_{\mathcal H}S_i^*h,k\rangle=\langle S_i^*h,P_{\mathcal H}k\rangle=\langle h,S_iP_{\mathcal H}k\rangle=\langle h,P_{\mathcal H}V_i k\rangle=\langle V_i^*h,k\rangle,$$ which implies that $S_i^*=V_{i}^{*}\mid_{\mathcal H}$ for $i=1,\ldots,n.$
	
It is also easy to verify the converse part. This completes the proof.
\end{proof}

The next Proposition gives the different characterization of the fundamental operators $E_i$'s of a $\Gamma_n$-contraction $(S_1,\ldots,S_n)$ which are described in \cite[Theorem $2.1$]{apal} and \cite[Lemma $4.1$]{tirt}.

\begin{prop}\label{Another E_i}
	The fundamental operators $E_i$ of a $\Gamma_n$-contraction $(S_1,\ldots,S_n)$ is the unique bounded linear operator $X_i$ on $D_{S_n}$ that satisfies the following operator equation: 
	
	\begin{equation}\label{fund equ} D_{S_n}S_i=X_{i}D_{S_n}+X_{n-i}^*D_{S_n}S_n~~{\rm{for}}~~i=1,\ldots,(n-1).\
	\end{equation}
	
\end{prop}
We will now prove the necessary conditions for existence of a $\Gamma_n$-isometry of $\Gamma_n$-contractions $(S_1,\ldots,S_n).$
\begin{thm}
	Let $(S_1,\ldots,S_n)$ be a $\Gamma_n$-contractions with fundamental operators $E_i \in \mathcal B(D_{S_n})$ and $\gamma_i=\frac{n-i}{n},$ for $i=1,\ldots,n-1.$ Each of the following conditions is necessary condition for existence of a $\Gamma_n$-isometry for $(S_1,\ldots,S_n):$
	\begin{enumerate}
		\item The $(n-1)$-tuple $(\gamma_1E_1,\ldots, \gamma_{n-1}E_{n-1})$ has joint  dilation to a commuting $(n-1)$-tuple of subnormal operators $(\gamma_1\tilde{S}_1,\ldots, \gamma_{n-1}\tilde{S}_{n-1}),$ that is, there exists an isometry $\Theta$ of $\mathcal D_{S_n}$ into a large Hilbert space $\mathcal{\tilde{H}}$ such that $\gamma_iE_i=\Theta^*\gamma_i\tilde{S}_i\Theta$ for $i=1,\ldots,(n-1),$ where  $(\gamma_1\tilde{S}_1,\ldots, \gamma_{n-1}\tilde{S}_{n-1})$ can be extended to a commuting $(n-1)$-tuple of normal operators $(\gamma_1U_1,\ldots, \gamma_{n-1}U_{n-1})$ whose joint spectrum $\sigma(\gamma_1U_1,\ldots, \gamma_{n-1}U_{n-1})$ lies in $\Gamma_{n-1}.$
		
		\item For $i=1,\ldots,(n-1),$ $(E_i^*D_{S_n}S_i-E_{n-i}^*D_{S_n}S_{n-i})_{|_{\ker D_{S_n}}}=0.$
		
		\item For $i=1,\ldots,(n-1),$ $(E_i^*E_{n-i}^*-E_{n-i}^*E_i^*)D_{S_n}{S_n}_{|_{\ker D_{S_n}}}=0.$
		
	\end{enumerate}
	
\end{thm}
\begin{proof}
Suppose $(V_1,\ldots,V_n)$ on $\mathcal K$ is a $\Gamma_n$-isometry of a $\Gamma_n$ contraction $(S_1,\ldots,S_n),$ therefore by Proposition $2.4$ we have 
	
	$$\mathcal K={\overline{\rm{span}}}\{V_1^{m_1}\ldots V_{n}^{m_{n}}h:h\in\mathcal H~{\rm{and}}~m_1,\ldots,m_n \in \mathbb N\cup\{0\}\}.$$
	Also, from Proposition \eqref{isometric dilation2}, we observe that 
	
	$$(V_1^*,\ldots,V_{n}^*)\mid_{\mathcal H}=(S_1^*,\ldots,S_n^*).$$
	The $2\times 2$ block operator matrix representation corresponding to the decomposition $\mathcal K=\mathcal H \oplus (\mathcal K \ominus \mathcal H)$ are of the form \begin{equation}\label{2 by 2 operator matrix}V_i=\begin{pmatrix}
			S_i & 0\\C_i & \tilde{S}_i
		\end{pmatrix}~~{\rm{for}}~~i=1,\ldots,n.\end{equation} 
	Since $V_n$ is isometry, the entries $S_n, C_n$ and $\tilde{S}_n$ of its $2\times 2$ block operator matrix satisfy the following operator identities:
	\begin{equation}\label{S_N}
		S_n^*S_n+C_n^*C_n=I_{\mathcal H} ~~{\rm{and}}~~\tilde{S}_n^*\tilde{S}_n=I_{\mathcal H}.
	\end{equation}
	From \eqref{S_N} we conclude that there exist an isometry $\Theta:\mathcal D_{S_n} \to \mathcal K \ominus \mathcal H$ such that  
	\begin{equation}\label{theta}
		\Theta D_{S_n}=C_n.
	\end{equation}
	As $(V_1,\ldots,V_n)$  is a $\Gamma_n$-isometry, from part $(2)$ of Theorem \eqref{Gamma_n isometry} we have $V_i=V_{n-i}^*V_n$ for $i=1,\ldots,(n-1).$ Thus, from \eqref{2 by 2 operator matrix} we get 
	$$
	\begin{pmatrix}
		S_i & 0\\C_i & \tilde{S}_i
	\end{pmatrix}=\begin{pmatrix}
		S_{n-i}^* & C_{n-i}^*\\0 & \tilde{S}_{n-i}^*
	\end{pmatrix}\begin{pmatrix}
		S_n & 0\\C_n & \tilde{S}_n
	\end{pmatrix}=\begin{pmatrix}
		S_{n-i}^*S_n+C_{n-i}^*C_n & C_{n-i}^*\tilde{S}_n\\\tilde{S}_{n-i}^*C_n & \tilde{S}_{n-i}^*\tilde{S}_n
	\end{pmatrix},
	$$
	which implies that 
	\begin{equation}\label{E_1}
		S_i-S_{n-i}^*S_n=C_{n-i}^*C_n,~~C_{n-i}^*\tilde{S}_n=0,~~C_i=\tilde{S}_{n-i}^*C_n ~~{\rm{and}}~~\tilde{S}_i=\tilde{S}_{n-i}^*\tilde{S}_n ~~{\rm{for}}~~i=1,\ldots,(n-1).
	\end{equation}
	From \eqref{E_1} and \eqref{theta}, it follows immediately that 
	$$
	D_{S_n}E_iD_{S_n}=S_i-S_{n-i}^*S_n=C_{n-i}^*C_n=C_n^*\tilde{S}_{i}C_n=D_{S_n}\Theta^*\tilde{S}_i\Theta D_{S_n}~~{\rm{for}}~~i=1,\ldots,(n-1).
	$$
	By uniqueness of the fundamental operators from, we obtain $$E_i=\Theta^*\tilde{S}_i\Theta~~~~{\rm{for}}~~i=1,\ldots,(n-1).$$
	We will use Theorem $2.2$ to show that $(\tilde{S}_1,\ldots,\tilde{S}_n)$ is a $\Gamma_n$-isometry. From \eqref{S_N} and \eqref{E_1}, we get $$\tilde{S}_i=\tilde{S}_{n-i}^*\tilde{S}_n ~~{\rm{and}}~~\tilde{S}_n^*\tilde{S}_n=I_{\mathcal H}.$$ Hence it is enough to show that  $(\gamma_1\tilde{S}_1,\ldots, \gamma_{n-1}\tilde{S}_{n-1})$ is a $\Gamma_{n-1}$-contraction. Since $\Gamma_{n-1}$ is a polynomially convex, therefore to show  $(\gamma_1\tilde{S}_1,\ldots, \gamma_{n-1}\tilde{S}_{n-1})$ is a $\Gamma_{n-1}$-contraction, it suffices to work with polynomials rather than the full algebra $\mathcal O(\Gamma_{n-1})$(See Oka-Weil Approximation Theorem  \cite[Page 84]{Gamelin}). As $(V_1,\ldots,V_n)$ is a $\Gamma_n$-isometry, from part $ (2)$ of Theorem \eqref{Gamma_n isometry}, we have $(\gamma_1V_1,\ldots,\gamma_{n-1}V_{n-1})$ is a $\Gamma_{n-1}$-contraction. Therefore, for any polynomial $p$, we get 
	\begin{align*}
		\|p(\gamma_1\tilde{S}_1,\ldots, \gamma_{n-1}\tilde{S}_{n-1})\| &= \|P_{\mathcal K \ominus \mathcal H}p(\gamma_1V_1,\ldots,\gamma_{n-1}V_{n-1})_{\mid_{\mathcal K \ominus \mathcal H}}\|\\& \leq \|p(\gamma_1V_1,\ldots,\gamma_{n-1}V_{n-1})\| \\&\leq \|p\|_{\infty,\Gamma_{n-1}},
	\end{align*}
	which implies that $(\gamma_1\tilde{S}_1,\ldots, \gamma_{n-1}\tilde{S}_{n-1})$ is a $\Gamma_{n-1}$-contraction. This shows that  $(\tilde{S}_1,\ldots,\tilde{S}_n)$ is a $\Gamma_n$-isometry. Therefore, $(\tilde{S}_1,\ldots,\tilde{S}_n)$ has a $\Gamma_n$-unitary extension, say $(U_1,\ldots,U_n).$ By definition of $\Gamma_n$-unitary, the joint spectrum $\sigma(U_1,\ldots,U_n)$ lies in the distinguish boundary of $\Gamma_n$ and $U_1,\ldots,U_n$ are normal operators. Note that $(\gamma_1U_1,\ldots, \gamma_{n-1}U_{n-1})$ is a $\Gamma_{n-1}$-contraction, because $(U_1,\ldots,U_n)$ is a 
	$\Gamma_n$-unitary. Therefore, the joint spectrum of commuting $(n-1)$ normal operators $(\gamma_1U_1,\ldots, \gamma_{n-1}U_{n-1})$ is contained in $\Gamma_{n-1},$ and item $(1)$ follows.
	
	Since $V_i$ and $V_{n-i}$ commute, by equating $(2,1)$- entry of $V_iV_{n-i}$ with the $(2,1)$- entry of $V_{n-i}V_{i},$ we observe that 
	\begin{equation}\label{Ci}
		C_iS_{n-i}+\tilde{S}_iC_{n-i}=C_{n-i}S_i+\tilde{S}_{n-i}C_i ~~{\rm{for}}~~i=1,\ldots,(n-1).
	\end{equation}
	Hence from \eqref{E_1} and \eqref{Ci}, we have
	\begin{equation}\label{Cn}
		{S}_{i}^*C_nS_{i}-\tilde{S}_{n-i}^*C_nS_{n-i}=(\tilde{S}_i\tilde{S}_i^*-\tilde{S}_{n-i}\tilde{S}_{n-i}^*)C_n ~~{\rm{for}}~~i=1,\ldots,(n-1).
	\end{equation}
	
	By multiplying $\Theta^*$ on the left side of \eqref{Cn} and using \eqref{theta} and the formulas for $E_i$ for $i=1,\ldots,(n-1),$ we get
	$$
	E_i^*D_{S_n}S_i-E_{n-i}^*D_{S_n}S_{n-i}=\Theta^*(\tilde{S}_i\tilde{S}_i^*-\tilde{S}_{n-i}\tilde{S}_{n-i}^*)\Theta D_{S_n} ~~{\rm{for}}~~i=1,\ldots,(n-1).
	$$
	This completes the proof of part $(2).$
	
	We use Proposition \eqref{Another E_i} to prove the part $(3).$ Note that for $i=1,\ldots,(n-1)$
	\begin{align}\label{part 3}
		\Theta^*(\tilde{S}_i\tilde{S}_i^*-\tilde{S}_{n-i}\tilde{S}_{n-i}^*)\Theta D_{S_n}\nonumber &=E_i^*D_{S_n}S_i-E_{n-i}^*D_{S_n}S_{n-i}\\\nonumber&=E_i^*(E_{i}D_{S_n}+E_{n-i}^*D_{S_n}S_n)-E_{n-i}^*(E_{n-i}D_{S_n}+E_{i}^*D_{S_n}S_n)\\ &=\nonumber(E_i^*E_{i}-E_{n-i}^*E_{n-i})D_{S_n}+(E_i^*E_{n-i}^*-E_{n-i}^*E_{i}^*)D_{S_n}S_n.
	\end{align}
	This completes the proof of part $(3).$ Also, from above observations, we conclude that the part $(2)$ and part $(3)$ are equivalent.
	
\end{proof}

\subsection{Sufficient conditions for existence of a $\Gamma_n$-isometry:}
\paragraph{The following theorem gives the sufficient conditions for unitary dilation of $\Gamma_n$-contractions. }
\begin{thm}\cite[Theorem $6.6$]{apal}\label{main dilation}
Suppose $(S_1,\ldots,S_n)$ is a $\Gamma_n$-contraction acting on a Hilbert space $\mathcal H.$  The commuting $(n-1)$-tuple of fundamental operators $(E_1,\ldots,E_{n-1})$ and $(F_1,\ldots,F_{n-1})$ of $(S_1,\ldots,S_n)$ and  $(S_1^{*},\ldots,S_{n}^{*})$  respectively obey the  conditions
$E_lE_{n-k}^{*}-E_kE_{n-l}^{*}=E_{n-k}^{*}E_l-E_{n-l}^{*}E_k$ and $F_l^{*}F_{n-k}-F_k^{*}F_{n-l}=F_{n-k}F_{l}^{*}-F_{n-l}F_{k}^{*}.$ Let $\mathcal K=\ldots \oplus \mathcal D_{S_n}\oplus \mathcal D_{S_n}\oplus \mathcal D_{S_n}\oplus \mathcal H \oplus \mathcal D_{S_n^*}\oplus \mathcal D_{S_n^*}\oplus \mathcal D_{S_n^*}\oplus \ldots$ and let $(R_1,\ldots,R_{n-1},U)$ be a $\Gamma_n$-contraction defined on $\mathcal K$ by
\scriptsize\begin{equation}R_i=\left[
\begin{array}{cccc|c|cccc}
 \ddots & \vdots & \vdots & \vdots & \vdots & \vdots & \vdots & \vdots &\iddots \\
 \ldots & E_i & E_{n-i}^* & 0 & 0 & 0 &0 & 0 & \ldots \\
 \ldots & 0 & E_i & E_{n-i}^* & 0 & 0 & 0 &0  & \ldots\\
 \ldots & 0 & 0& E_i & E_{n-i}^*D_{S_n} & - E_{n-i}^*S_n^* & 0 & 0   & \ldots\\
  \hline
  \ldots& 0 & 0 & 0 & S_i & D_{S_n^*}F_{n-i} & 0 & 0 & \ldots\\
  \hline
 \ldots &0 & 0 & 0 & 0 & F_{i}^{*} & F_{n-i} & 0 & \ldots\\
\ldots &0 &0 & 0 & 0 & 0 & F_{i}^{*} & F_{n-i}  & \ldots\\
  \iddots & \vdots & \vdots & \vdots & \vdots & \vdots & \vdots & \vdots &\ddots
\end{array}
\right]
\end{equation}\normalsize and \scriptsize\begin{equation}U=\left[
\begin{array}{cccc|c|cccc}
  \ddots & \vdots & \vdots & \vdots & \vdots & \vdots & \vdots & \vdots &\iddots \\
  \ldots & 0 & I & 0 & 0 & 0 &0 & 0 & \ldots \\
  \ldots & 0 & 0 & I & 0 & 0 & 0 &0  & \ldots\\
  \ldots & 0 & 0& 0 & D_{S_n} & -S_n^* & 0 & 0   & \ldots\\
  \hline
  \ldots& 0 & 0 & 0 & P & D_{S_n^*} & 0 & 0 & \ldots\\
  \hline
  \ldots &0 & 0 & 0 & 0 & 0 & I & 0 & \ldots\\
\ldots &0 &0 & 0 & 0 & 0 & 0 & I  & \ldots\\
  \iddots & \vdots & \vdots & \vdots & \vdots & \vdots & \vdots & \vdots &\ddots
\end{array}
\right].
\end{equation}\normalsize
 Then $(R_1,\ldots,R_{n-1},U)$ is a minimal $\Gamma_n$-unitary dilation of $(S_1,\ldots,S_n).$
\end{thm}

\section{{\bf{ $\Gamma_3$-contractions with some special forms:}}}
Consider $(T_1,T_2,V_3)$ is a commuting $3$-tuple of bounded operators on some Hilbert space $\mathcal H,$ where $T_1$ and $T_2$ are contractions on $\mathcal H$ and $V_3$  is an isometry on $\mathcal H$. Suppose $T_1,T_2$ and $V_3$ are of the above form, then we will show that there is a $\Gamma_3$-contraction of the form $\left(\frac{T_1+T_2+V_3}{3},\frac{T_1T_2+T_2V_3+V_3T_1}{3},T_1T_2V_3 \right)$ 
 which always dilate.
\begin{lem}\label{iso}
Suppose $T_1,T_2$ and $V_3$ as above. Then $(V_1,V_2,V_3\oplus I_{\mathcal K\ominus \mathcal H})$ is an isometric dilation of $(T_1,T_2,V_3),$ where $(V_1,V_2)$ is the isometric dilation of $(T_1,T_2)$ on $\mathcal K$  and $I_{\mathcal K\ominus \mathcal H}$ is the identity operator on $\mathcal K\ominus \mathcal H.$
\end{lem}
\vskip.3cm
\begin{proof}
Since $(T_1,T_2)$ are commuting contractions on $\mathcal H,$ by Ando's theorem, there exist a Hilbert space $\mathcal K$ containing $\mathcal H$ as a subspace such that $T_i=P_{\mathcal H}{V_i}_{|_{\mathcal H}}$ for $i=1,2.$ As $V_1,V_2$ and $V_3\oplus I_{\mathcal K\ominus \mathcal H}$ are isometries, we conclude that $(V_1,V_2,V_3\oplus I_{\mathcal K\ominus \mathcal H})$ is an isometric dilation of $(T_1,T_2,V_3).$ This completes the proof.
\end{proof}
\begin{lem}\label{gamma iso}
Let $(V_1,V_2,V_3)$ be  a commuting $3$-tuple of isometries on some Hilbert space $\mathcal H.$  Then $$\left(S_1=\frac{V_1+V_2+V_3}{3},S_2=\frac{V_1V_2+V_2V_3+V_3V_1}{3},S_3=V_1V_2V_3\right)$$ is $\Gamma_3$-isometry.
\end{lem}
\begin{proof}
 Clearly, $S_3$ is an isometry as $V_1,V_2$ and $V_3$ are isometries. It is easy to check that $S_1=S_2^*S_3$ and $S_2=S_1^*S_3.$ Now all we require to show that $(\frac{2}{3}S_1,\frac{1}{3}S_2)$ is a $\Gamma_2$-contraction which is equivalent to $$(2-\frac{2}{3}S_1)^* (2-\frac{2}{3}S_1)-(2\frac{1}{3}S_2-\frac{2}{3}S_1)^*(2\frac{1}{3}S_2-\frac{2}{3}S_1)\geq 0.$$
Note that 
\begin{align}\label{A}
A\nonumber&=(2-\frac{2}{3}S_1)^* (2-\frac{2}{3}S_1)\\&=\frac{4}{81}\left(84I-9V_1-9V_2-9V_3-9V_1^*-9V_2^*-9V_3^*+V_2^*V_1+V_3^*V_1+V_1^*V_2+V_3^*V_2+V_1^*V_3+V_2^*V_3
\right)\end{align}
and
\begin{align}\label{B}
B\nonumber&=(2\frac{1}{3}S_2-\frac{2}{3}S_1)^*(2\frac{1}{3}S_2-\frac{2}{3}S_1)\\&=\frac{4}{81}\left(6I-2V_1-2V_2-2V_3-2V_1^*-2V_2^*-2V_3^*+2V_2^*V_1+2V_3^*V_1+2V_1^*V_2+2V_3^*V_2+2V_1^*V_3
+2V_2^*V_3 -C \right),
\end{align}
where $C=V_2^*V_1V_3+V_3^*V_1^*V_2+V_1^*V_2V_3+V_3^*V_2^*V_1
+V_3^*V_1V_2+V_2^*V_1^*V_3.$ From \eqref{A} and \eqref{B}, we have
\begin{align*}
A-B&=\frac{4}{81}\{78I-7(V_1+V_2+V_3+V_1^*+V_2^*+V_3^*)
-V_2^*V_1-V_3^*V_1-V_1^*V_2-V_3^*V_2-V_1^*V_3-V_2^*V_3\\&+V_2^*V_1V_3+V_3^*V_1^*V_2+V_1^*V_2V_3+V_3^*V_2^*V_1
+V_3^*V_1V_2+V_2^*V_1^*V_3\}\\&=\frac{4}{81}\{24I+7(I-V_1)^*(I-V_1)+7(I-V_2)^*(I-V_2)+7(I-V_3)^*(I-V_3)+(I-V_2^*V_1)^*(I-V_2^*V_1)\\&+(I-V_3^*V_1)^*(I-V_3^*V_1)+(I-V_2^*V_3)^*(I-V_2^*V_3)+(I+V_2^*V_1V_3)^*(I+V_2^*V_1V_3)
\\&+(I+V_1^*V_2V_3)^*(I+V_1^*V_2V_3)+(I+V_3^*V_1V_2)^*(I+V_3^*V_1V_2)\}\geq 0,
\end{align*}
Thus, from item $2$ of the Theorem $2.2$, we conclude that $(S_1,S_2,S_3)$ is a $\Gamma_3$-isometry.

\end{proof}
\begin{thm}\label{gamma 3 isometry}
Suppose $(T_1,T_2,V_3)$ is  a commuting $3$-tuple of  operators acting on some Hilbert space $\mathcal H$ with $T_1$ and $T_2$ are contractions and $V_3$  is an isometry. Then $\left(\frac{T_1+T_2+V_3}{3},\frac{T_1T_2+T_2V_3+V_3T_1}{3},T_1T_2V_3 \right)$ is $\Gamma_3$-contraction and $\left(\frac{T_1+T_2+V_3}{3},\frac{T_1T_2+T_2V_3+V_3T_1}{3},T_1T_2V_3\right)$ has a $\Gamma_3$-isometric dilation.
\end{thm}
\begin{proof}
Let $(V_1,V_2)$ be an Ando isometric dilation of $(T_1,T_2).$ Then there exist a Hilbert space $\mathcal K\supseteq \mathcal H$ as a subspace such that $T_i=P_{\mathcal H}{V_i}_{|_{\mathcal H}}$ for $i=1,2.$ Now from Lemma \ref{iso}, it implies that $(V_1,V_2,V_3\oplus I_{\mathcal K\ominus \mathcal H})$ is an isometric dilation of $(T_1,T_2,V_3).$  Set  $$\left(S_1=\frac{V_1+V_2+V_3\oplus I_{\mathcal K\ominus \mathcal H}}{3},S_2=\frac{V_1V_2+V_2(V_3\oplus I_{\mathcal K\ominus \mathcal H})+(V_3\oplus I_{\mathcal K\ominus \mathcal H})V_1}{3},S_3=V_1V_2(V_3\oplus I_{\mathcal K\ominus \mathcal H})\right).$$ Since $\left(S_1,S_2,S_3\right)$ is $\Gamma_3$-isometry, from item $4$ of Theorem $2.2$ it implies that $\left(S_1,S_2,S_3\right)$ is $\Gamma_3$-contraction and $S_3$ is an isometry. We all knows that $\Gamma_3$ is polynomially convex therefore we can use Oka-Weil Approximation Theorem (See \cite[Pg. 84 Theorem ]{Gamelin}). So, it is enough to work with polynomials rather than the full algebra $\mathcal O(\Gamma_3).$ Let $p$ be any polynomial we have
\begin{align*}
\|p(\frac{T_1+T_2+V_3}{3},\frac{T_1T_2+T_2V_3+V_3T_1}{3},T_1T_2V_3 )\|&=\|P_{\mathcal H}p(S_1,S_2,S_3)_{|_\mathcal H}\|\\&\leq \|p(S_1,S_2,S_3)\|\\ & \leq \|p\|_{\infty,\Gamma_3}.
\end{align*}
This shows that $\left(\frac{T_1+T_2+V_3}{3},\frac{T_1T_2+T_2V_3+V_3T_1}{3},T_1T_2V_3 \right)$ is $\Gamma_3$-contraction. This completes the proof.

\end{proof}

To work with more tractable examples, in this section we will consider $\Gamma_3$-contractions $(S_1,S_2,S_3)$, where $S_3$ is a partial isometry. We begin with the following result which  appears in \cite[Proposition 3.2]{Ball}.
\begin{prop}\label{Hari}
Let  $(T_1,T_2)$ be a pair of contractions on $\mathcal H$ and $T$ be a partial isometry on $\mathcal H.$ Suppose there exist two operators $E_1,E_2$ in $\mathcal B(\mathcal D_{T})$ such that $$T_i-T_{3-i}^*T=D_{T}E_iD_{T}~~{\rm{for}~~} i=1,2.$$ Then 
\begin{enumerate}
\item $\ker T$ is jointly invariant under $(T_1,T_2)$ and,
\item if we represent the restriction $(T_1,T_2)_{|_{\ker T}} $ by $(D_1,D_2),$ then 
\begin{enumerate}
\item[(a)] $E_1^*E_1-E_1E_1^*=E_2^*E_2-E_2E_2^*$ if and only if $D_1^*D_1-D_1D_1^*=D_2^*D_2-D_2D_2^*,$ and 
\item[(b)]$E_1E_2=E_2E_1$ if and only if $D_1D_2=D_2D_1.$ 
\end{enumerate}
\end{enumerate}
\end{prop}
The following corollary is a simple outcome of the above Proposition.

\begin{cor}\label{gamma3}
Let $(S_1,S_2,S_3)$ be a $\Gamma_3$-contraction acting on a Hilbert space $\mathcal H$ having fundamental operators $E_1$ and $E_2$ with $(S_1,S_2)$ a $2$-tuple of commuting contractions on $\mathcal H.$ If $S_3$ is a partial isometry, then \begin{enumerate}
\item $\ker S_3$ is jointly invariant under $(S_1,S_2)$ and,
\item 
 $E_1E_2=E_2E_1$  and 

\item if we denote the restriction $(S_1,S_2)_{|_{\ker S_3}} $ by $(D_1,D_2),$ then  $E_1^*E_1-E_1E_1^*=E_2^*E_2-E_2E_2^*$ if and only if $D_1^*D_1-D_1D_1^*=D_2^*D_2-D_2D_2^*.$

\end{enumerate}
\end{cor}
We will now produce an example of $\Gamma_3$ contraction $(S_1,S_2,S_3)$ which satisfies all conditions in Corollary \eqref{gamma3}.
\begin{example}\label{example1}
Consider the following triple of commuting operators on $H^2\oplus H^2:$ 
$$(T_1,T_2,T_3)=\left( \left(\begin{smallmatrix} 0 & 0 \\I_{H^2} & 0 \end{smallmatrix}\right), \left(\begin{smallmatrix} T_z & 0 \\  0 & T_z \end{smallmatrix}\right), \left(\begin{smallmatrix} I_{H^2} & 0 \\ 0 & I_{H^2}\end{smallmatrix}\right)\right),$$ where $T_z$ is a multiplication operator on $H^2.$ Let 
$$S_1=\frac{1}{3}(T_1+T_2+V_3)=\frac{1}{3}\left(\begin{smallmatrix} I_{H^2}+T_z & 0 \\ I_{H^2} & I_{H^2}+T_z \end{smallmatrix}\right),S_2=\frac{1}{3}(T_1T_2+T_2T_3+T_1T_3)=\frac{1}{3}\left(\begin{smallmatrix} T_z & 0\\I_{H^2}+T_z  & T_z \end{smallmatrix}\right)$$ and $ S_3=T_1T_2V_3=\left(\begin{smallmatrix} 0 & 0 \\ T_z & 0  \end{smallmatrix}\right).$ Therefore by using Theorem \eqref{gamma 3 isometry} we conclude that the triple 
$\left(S_1,S_2,S_3 \right)$ is $\Gamma_3$-contraction and  has a $\Gamma_3$-isometric dilation. Since $T_z$ is an isometry, one can easily verify that $S_3$ is a partial isometry. First we will compute the defect operator for $S_3:$
$$D^2_{S_3}=\left(\begin{smallmatrix} I_{H^2} & 0 \\ 0 & I_{H^2} \end{smallmatrix}\right)-\left(\begin{smallmatrix} 0 & 0 \\ T_z & 0  \end{smallmatrix}\right)^*\left(\begin{smallmatrix} 0 & 0 \\ T_z & 0  \end{smallmatrix}\right)=\left(\begin{smallmatrix} 0 & 0 \\  0 & I_{H^2} \end{smallmatrix}\right)=D_{S_3}.$$ 
Let us consider $$(A_1,A_2)=(\frac{I_{H^2}+T_z}{3},\frac{T_z}{3}).$$ Notice that 
\begin{align*}
S_1-S_2^*S_3&=\frac{1}{3}\left(\begin{smallmatrix} I_{H^2}+T_z & 0 \\I_{H^2} &  I_{H^2}+T_z \end{smallmatrix}\right)-\frac{1}{3}\left(\begin{smallmatrix} T_z & 0 \\I_{H^2}+T_z &  T_z \end{smallmatrix}\right)^*\left(\begin{smallmatrix} 0 & 0 \\ T_z & 0  \end{smallmatrix}\right)\\&=\frac{1}{3}\left(\begin{smallmatrix} I_{H^2}+T_z & 0 \\I_{H^2} & I_{H^2}+T_z \end{smallmatrix}\right)-\frac{1}{3}\left(\begin{smallmatrix} I_{H^2}+T_z & 0 \\I_{H^2} & 0  \end{smallmatrix}\right)\\&=\frac{1}{3}\left(\begin{smallmatrix} 0 & 0 \\ 0 & I_{H^2}+T_z \end{smallmatrix}\right)=\left(\begin{smallmatrix} 0 & 0 \\ 0 & I_{H^2} \end{smallmatrix}\right)\frac{1}{3}\left(\begin{smallmatrix} 0 & 0 \\ 0 & I_{H^2}+T_z \end{smallmatrix}\right)\left(\begin{smallmatrix} 0 & 0 \\ 0 & I_{H^2} \end{smallmatrix}\right)=D_{S_3}A_1D_{S_3}
\end{align*}
and
\begin{align*}
S_2-S_1^*S_3&=\frac{1}{3}\left(\begin{smallmatrix} T_z & 0 \\I_{H^2}+T_z  & T_z \end{smallmatrix}\right)-\frac{1}{3}\left(\begin{smallmatrix}I_{H^2}+ T_z & 0 \\I_{H^2}  & I_{H^2}+T_z \end{smallmatrix}\right)^*\left(\begin{smallmatrix} 0 & 0 \\T_z & 0  \end{smallmatrix}\right)\\&=\frac{1}{3}\left(\begin{smallmatrix} T_z & 0 \\I_{H^2}+T_z  & T_z \end{smallmatrix}\right)-\frac{1}{3}\left(\begin{smallmatrix} T_z & 0 \\I_{H^2}+T_z & 0  \end{smallmatrix}\right)\\&=\frac{1}{3}\left(\begin{smallmatrix} 0 & 0 \\ 0 & T_z \end{smallmatrix}\right)=\left(\begin{smallmatrix} 0 & 0 \\ 0 & I_{H^2} \end{smallmatrix}\right)\frac{1}{3}\left(\begin{smallmatrix} 0 & 0 \\ 0 & T_z \end{smallmatrix}\right)\left(\begin{smallmatrix} 0 & 0 \\ 0 & I_{H^2} \end{smallmatrix}\right)=D_{S_3}A_2D_{S_3}
\end{align*}
which implies that $(A_1,A_2)$ are same as the fundamental operators $(E_1,E_2)$ of $\left(S_1,S_2,S_3 \right).$ Notice that 
$$A_1^*A_1-A_1A_1^*=\frac{1}{9}(I-T_zT_z^*)= A_2^*A_2-A_2A_2^*.$$ One can easily verify that the conditions $(1)$ and $(2)$ of Corollary \eqref{gamma3} are satisfied. To verify the condition $(3)$ of Corollary \eqref{gamma3}, we need to show that $D_1^*D_1-D_1D_1^*=D_2^*D_2-D_2D_2^*.$ It is also easy to check that $\ker S_3=\{0\}\oplus H^2$ and 
$(S_1,S_2)_{|_{\ker S_3}}=(D_1,D_2)=(\frac{I_{H^2}+T_z}{3},\frac{T_z}{3}),$ which implies 
$$D_1^*D_1-D_1D_1^*=D_2^*D_2-D_2D_2^*=\frac{1}{9}(I-T_zT_z^*).$$ Thus, the condition $(3)$ of Corollary \eqref{gamma3} is satisfied.

\end{example}

The subsequent example of a $\Gamma_3$-contraction has a $\Gamma_3$-isometric dilation for which condition $(3)$ in Corollary \eqref{gamma3} fails, that is,  $$E_1^*E_1-E_1E_1^*\neq E_2^*E_2-E_2E_2^*.$$ In other words, we can conclude that the set of sufficient conditions for existence of a $\Gamma_3$-isometric dilation in Theorem \eqref{main dilation} breaks down in general to be necessary, even when the $\Gamma_3$-contraction $(S_1,S_2,S_3)$ has the following forms:
\begin{enumerate}
\item $(S_1,S_2)$ is a $2$-tuple of commuting contractions on some Hilbert space $\mathcal H,$ and 
\item $S_3$ is a partial isometry on $\mathcal H.$
\end{enumerate} 
\begin{example}\label{example2}
Let us consider the following pair of contractions on $H^2\oplus H^2 \oplus H^2$:
$$(T_1,T_2)=\left( \left(\begin{smallmatrix} 0 & 0 &I_{H^2}\\0 & 0 &0\\I_{H^2} & 0 & 0\end{smallmatrix}\right), \left(\begin{smallmatrix} T_z & 0 & 0\\ 0 & 0 & 0\\0 & 0 & T_z \end{smallmatrix}\right)\right),$$ where $T_z$ is a multiplication operator on $H^2.$ Let  
\small{$$S_1=\frac{1}{3}(T_1+T_2+V_3)=\frac{1}{3}\left(\begin{smallmatrix} I_{H^2}+T_z & 0 & I_{H^2}\\ 0 & T_z & 0\\I_{H^2} & 0 & I_{H^2}+T_z \end{smallmatrix}\right),S_2=\frac{1}{3}(T_1T_2+T_2V_3+T_1V_3)=\frac{1}{3}\left(\begin{smallmatrix} T_z & 0 & I_{H^2}+T_z\\ 0 & 0 & 0\\I_{H^2}+T_z & 0 & T_z \end{smallmatrix}\right)$$} and $ S_3=T_1T_2V_3=\left(\begin{smallmatrix} 0 & 0 & T_z\\ 0 & 0 & 0\\T_z & 0 & 0 \end{smallmatrix}\right),$ where $V_3=\left(\begin{smallmatrix} I_{H^2} & 0 & 0\\ 0 & T_z & 0\\0 & 0 & I_{H^2} \end{smallmatrix}\right).$  Therefore by Theorem \eqref{gamma 3 isometry} the triple 
$\left(S_1,S_2,S_3 \right)$ is $\Gamma_3$-contraction which  has a $\Gamma_3$-isometric dilation. As $T_z$ is an isometry, one can easily check that $S_3$ is a partial isometry. Therefore we can apply the conclusion of the Proposition \eqref{Hari} for this triple. The first step for computing the fundamental operators $E_1$ and $E_2,$ is to calculate the defect operator for $S_3:$
$$D^2_{S_3}=\left(\begin{smallmatrix} I_{H^2} & 0 & 0\\ 0 & I_{H^2} & 0\\0 & 0 & I_{H^2} \end{smallmatrix}\right)-\left(\begin{smallmatrix} 0 & 0 & T_z\\ 0 & 0 & 0\\T_z & 0 & 0 \end{smallmatrix}\right)^*\left(\begin{smallmatrix} 0 & 0 & T_z\\ 0 & 0 & 0\\T_z & 0 & 0 \end{smallmatrix}\right)=\left(\begin{smallmatrix} 0 & 0 & 0\\ 0 & I_{H^2} & 0\\0 & 0 & 0 \end{smallmatrix}\right)=D_{S_3}.$$ Set $$(A_1,A_2)=(\frac{T_z}{3},0).$$ The following computations
\begin{align*}
S_1-S_2^*S_3&=\frac{1}{3}\left(\begin{smallmatrix} I_{H^2}+T_z & 0 & I_{H^2}\\ 0 & T_z & 0\\I_{H^2} & 0 & I_{H^2}+T_z \end{smallmatrix}\right)-\frac{1}{3}\left(\begin{smallmatrix} T_z & 0 & I_{H^2}+T_z\\ 0 & 0 & 0\\I_{H^2}+T_z & 0 & T_z \end{smallmatrix}\right)^*\left(\begin{smallmatrix} 0 & 0 & T_z\\ 0 & 0 & 0\\T_z & 0 & 0 \end{smallmatrix}\right)\\&=\frac{1}{3}\left(\begin{smallmatrix} I_{H^2}+T_z & 0 & I_{H^2}\\ 0 & T_z & 0\\I_{H^2} & 0 & I_{H^2}+T_z \end{smallmatrix}\right)-\frac{1}{3}\left(\begin{smallmatrix} I_{H^2}+T_z & 0 & I_{H^2}\\ 0 & 0 & 0\\I_{H^2} & 0 & I_{H^2}+T_z \end{smallmatrix}\right)\\&=\frac{1}{3}\left(\begin{smallmatrix} 0 & 0 & 0\\ 0 & T_z & 0\\0 & 0 & 0 \end{smallmatrix}\right)\\&=\left(\begin{smallmatrix} 0 & 0 & 0\\ 0 & I_{H^2}& 0\\0 & 0 & 0 \end{smallmatrix}\right)\frac{1}{3}\left(\begin{smallmatrix} 0 & 0 & 0\\ 0 & T_z & 0\\0 & 0 & 0 \end{smallmatrix}\right)\left(\begin{smallmatrix} 0 & 0 & 0\\ 0 & I_{H^2} & 0\\0 & 0 & 0 \end{smallmatrix}\right)=D_{S_3}A_1D_{S_3}
\end{align*}
and
\begin{align*}
S_2-S_1^*S_3&=\frac{1}{3}\left(\begin{smallmatrix} T_z & 0 & I_{H^2}+T_z\\ 0 & 0 & 0\\I_{H^2}+T_z & 0 & T_z \end{smallmatrix}\right)-\frac{1}{3}\left(\begin{smallmatrix}I_{H^2}+ T_z & 0 & I_{H^2}\\ 0 & T_z & 0\\I_{H^2} & 0 & I_{H^2}+T_z \end{smallmatrix}\right)^*\left(\begin{smallmatrix} 0 & 0 & T_z\\ 0 & 0 & 0\\T_z & 0 & 0 \end{smallmatrix}\right)\\&=\frac{1}{3}\left(\begin{smallmatrix} T_z & 0 & I_{H^2}+T_z\\ 0 & 0 & 0\\I_{H^2}+T_z & 0 & T_z \end{smallmatrix}\right)-\frac{1}{3}\left(\begin{smallmatrix} T_z & 0 & I_{H^2}+T_z\\ 0 & 0 & 0\\I_{H^2}+T_z & 0 & T_z \end{smallmatrix}\right)\\&=\frac{1}{3}\left(\begin{smallmatrix} 0 & 0 & 0\\ 0 & 0 & 0\\0 & 0 & 0 \end{smallmatrix}\right)=D_{S_3}A_2D_{S_3}
\end{align*}
 implies that $(A_1,A_2)$ are the same as the fundamental operators $(E_1,E_2)$ of $\left(S_1,S_2,S_3 \right).$ Note that 
$$A_1^*A_1-A_1A_1^*=\frac{1}{9}(I-T_zT_z^*)\neq 0 ~~{\rm{while~~}} A_2^*A_2-A_2A_2^* =0.$$ 

One can easily verify that $\ker S_3=\{0\}\oplus H^2\oplus \{0\}$ and $(S_1,S_2)_{|_{\ker S_3}}=(D_1,D_2)=(\frac{T_z}{3},0).$ Also, we have $$D_1^*D_1-D_1D_1^*=\frac{1}{9}(I-T_zT_z^*)\neq 0 ~~{\rm{while~~}} D_2^*D_2-D_2D_2^* =0,$$
which shows that  the condition $(3)$ of Corollary \eqref{gamma3} is disobeyed.

\end{example}
\begin{rem}

In Example $2,$ we notice that if $(S_1,S_2)$ is a $2$-tuple of commuting contractions on  $H^2,$ and $S_3$ is a partial isometry on $H^2,$ then we have $$D_1^*D_1-D_1D_1^*\neq  D_2^*D_2-D_2D_2^*.$$  Therefore, it follows from condition $(3)$ of Corollary \eqref{gamma3} that  $$E_1^*E_1-E_1E_1^*\neq  E_2^*E_2-E_2E_2^*$$ This shows that the $\Gamma_3$-contraction $(S_1,S_2,S_3)$ does have $\Gamma_3$-isometric dilation, but it fails to satisfy the following condition: 
$$E_1^*E_1-E_1E_1^*\neq  E_2^*E_2-E_2E_2^*.$$ However, as has been mentioned, Example $2$ above appears that the condition $E_1^*E_1-E_1E_1^*\neq  E_2^*E_2-E_2E_2^*$ is not necessary for existing of a $\Gamma_3$-isometric dilation. Still now we are not able to locate an example of $\Gamma_n$-contraction for $n\geq 3,$  which fails to satisfy one of the necessary conditions. Thus, the existence of rational dilation for $\Gamma_n$-contraction, for $n\geq 3$ is still an open question.
\end{rem}

\textsl{Acknowledgements:}
The first named author thanks CSIR for financial support and the second named author thankfully acknowledges the financial support provided by Mathe-matical Research Impact Centric Support (MATRICS) grant, File no: MTR/2020/000493,by the Science and Engineering Research Board (SERB), Department of Science and Tech-nology (DST), Government of India. 
\vskip-1cm



\begin{thebibliography}{99}
\vskip-.4cm
\bibitem{A}   W. Arveson,
              \textit{Subalgebras of $C^*$-algebras, }
               Acta Math., {\bf 123} (1969), 141 -224.



\bibitem{AW}   W. Arveson,
              \textit{Subalgebras of $C^*$-algebras II,}
               Acta Math., {\bf 128} (1972), 271 -308.


\bibitem{agler} J. Agler,
               \textit{Rational dilation on an annulus,}
               Ann. of Math., {\bf 121} (1985), 537 - 563.


\bibitem{AHR} J. Agler, J. Harland, B. J. Raphael,
             \textit{Classical function theory, operator dilation theory, and machine computation on multiply-connected          domains,}
             Mem. Amer. Math. Soc. {\bf 191} (2008), 289 -- 312.

\bibitem{ALY} J. Agler, Z. A. Lykova, N. J. Young,
             \textit{Extremal holomorphic maps and the symmetrized bidisc,}
             Proc. London Math. Soc., {\bf 106 }(2013) 781--818.



\bibitem{AM} J. Agler, J. McCarthy,
             \textit{Pick Interpolation and Hilbert Function Spaces,}
             Graduate studies in mathematics {\bf 44,} Amer. Math. Soc., Providence, R.I. (2002).




\bibitem{young} J. Agler,  N. J. Young,
               \textit{Operators having the symmetrized bidisc as a spectral set,}
               Proc. Edinburgh Math. Soc., {\bf 43} (2000), 195 -210.




\bibitem{AY} J. Agler, N.J. Young,
             \textit{A commutant lifting theorem for a domain in $\mathbb C^2$ and spectral interpolation,}
             J. Funct. Anal. {\bf 161} (1999) 452--477.






\bibitem{Agler} J. Agler, N.J. Young,
               \textit{A model theory for $\Gamma$-contractions,}
                J. Operator Theory {\bf 49} (2003) 45--60.

\bibitem{JAgler} J. Agler, N.J. Young,
                \textit{The hyperbolic geometry of the symmetrized bidisc,}
                J. Geom. Anal. {\bf 14} (2004) 375--403.



\bibitem{TANDO} T. Ando,
              \textit{On a pair of commutative contractions,}
               Acta Sci Math {\bf 24} (1963), 88 -- 90.



\bibitem{Ando}  T. Ando,
               \textit{Structure of operators with numerical radius one,}
               Acta Sci. Math. (Szeged){\bf 34} (1973) 11--15.



\bibitem{BFK}H. Bercovici, C. Foias, L. Kerchy, B. Sz.-Nagy,
             \textit{Harmonic analysis of operators on Hilbert space,}
             Universitext, Springer, New York, (2010).

\bibitem{BPR} T. Bhattacharyya, S. Pal, S. Shyam Roy,
            \textit{Dilations of $\Gamma$-contractions by solving operator equations,}
             Adv. Math.{\bf 230} (2012), 577 -- 606.


\bibitem{tirt} T. Bhattacharyya,
\textit{The tetrablock as a spectral set, }
Indiana Univ. Math. J., {\bf 63}(6) (2014), 1601-1629.

\bibitem{BSS} T. Bhattacharyya, Sneh Lata, H. Sau,
              \textit{Admissible fundamental operators,}
              J . Math. Anal. Appl., {\bf 425} (2015), 983 --1003.


\bibitem{SS} S. Biswas, S. Shyam Roy,
             \textit{Functional models for $\Gamma_n$-contractions and characterization of $\Gamma_n$-isometries,}
             J. Func. Anal., {\bf 266} (2014), 6224 --6255.

\bibitem{Ball} J. A. Ball and  H. Sau, 
              \textit{Rational dilations of tetrablock contractions revisited,} 
              J.Functional Analysis,278(2020)1-15.












\bibitem{costara} C. Costara,
                 \textit{The symmetrized bidisc and Lempert?s theorem,}
                  Bull. London Math. Soc. {\bf 36} (2004) 656--662.

\bibitem{Ccostara}C. Costara,
                 \textit{On the spectral Nevanlinna-Pick problem,}
                 Studia Math., {\bf  170} (2005), 23--55.

\bibitem{curto} R. E. Curto,
               \textit{Applications of several complex variables to multiparameter spectral theory,}
               Surveys of Some Recent Results in Operator Theory, Vol. II, Pitman Res.
               Notes Math. Ser., Longman Sci. Tech., Harlow, {\bf 192} (1988), 25-90.





%



\bibitem{michel} M.  Dritschel, S. McCullough,
              \textit{The failure of rational dilation on        a triply connected domain,}
                J. Amer. Math. Soc., {\bf 18} (2005), 873- 918.



\bibitem{Zwonek} A. Edigarian, W. Zwonek, 
                \textit{Geometry of symmetrized polydisc,}
                 Archiv der Mathematik, {\bf 84} (2005), 364 -374.


\bibitem{Gamelin} T. Gamelin, 
                 \textit{Uniform Algebras,}
                  Prentice Hall, New Jersey, (1969).

\bibitem{Gorai} S. Gorai, J. Sarkar, 
               \textit{Characterizations of symmetrized polydisc,}
                To appear in Indian Jour. Pure and Applied Math., 
                available at arXiv:1503.03473v1 [math.CV] 11 Mar (2015).



%
%
%
%
%






\bibitem{cv} G. Misra, A. Pal and C. Varughese,   \textit{Contractivity and complete contractivity for   finite dimensional Banach spaces,} 
             Journal of operator theory, {\bf 1} (2019), 23-47. 



\bibitem{GM} G. Misra, N. S. N. Sastry,
             \textit{Contractive modules, extremal problems and curvature inequalities,}
             J. Funct. Anal., {\bf 88} (1990), 118 - 134.



\bibitem{sastry}  G. Misra, N. S. N. Sastry,
                \textit{Completely contractive modules and associated extremal problems,}
                J. Funct. Anal., {\bf 91 } (1990), 213 - 220.
%




%
%




\bibitem{apal} A.Pal , 
             \textit{On $\Gamma_n$-contractions and their     Conditional Dilations} 
             {{arXiv:1704.04508v2}}




\bibitem{pal3} S. Pal
               \textit{The failure of rational dilation on the symmetrized $n$-disk for any $n\geq 3.$, }
                arXiv:1712.05707 


\bibitem{pal4} S. Pal
               \textit{Canonical decomposition of operators associated with the symmetrized polydisc }
               arXiv:1708.00724


\bibitem{pal6} S. Pal
              \textit{Operator theory and distinguished varities in the symmetrized $n$-disk }
               arXiv:1708.06015


\bibitem{pisier} G. Pisier,
               \textit{ Introduction to  Operator Spaces Theory,}
                Cambridge Univ. Press, (2003).



\bibitem{paulsen} V. Paulsen,
                \textit{Completely Bounded Maps and Operator Algebras,}
                Cambridge  Univ. Press, (2002).
%


\bibitem{vern} V. Paulsen,
               \textit{Representations of Function Algebras, Abstract Operator Spaces and Banach Space Geometry, }
                 J. Funct. Anal., {\bf 109} (1992), 113 - 129.



\bibitem{patak} V. Ptak, N. J. Young, 
              \textit{A generalization of zero location theorem of Schur and Cohn,} 
             Trans. Inst. Electrical Electron. Automat. Control, {\bf 25} (1980), 978 -980.
%
%
%

%


\bibitem{Schur} I. Schur, 
               \textit{\"Uber Potenzreihen die im Innern des Einheitskreises beschr\"ankt sind,} 
               Jour. \"fur Math., I {\bf 147 } (1917), 205 --232 ; II{\bf 148} (1918), 122 -- 145.

\bibitem{bela} Bela Sz.-Nagy, 
             \textit{Sur les contractions de \,lespace de Hilbert,} 
             Acta Sci. Math., {\bf 15} (1953), 87 -92.





\bibitem{stout} E. L. Stout,
               \textit{Polynomial Convexity,}
               Progr. Math., Birkh\"auser Boston, Inc., Boston, MA, (2007).



\bibitem{Taylor}  J. L. Taylor, 
                \textit{The analytic-functional calculus for several commuting operators,} 
                Acta Math. {\bf 125} (1970), 1 --38.
                
\bibitem{jtaylor} J. L. Taylor, 
                  \textit{A joint spectrum for several commuting operators,}
                   J. Funct. Anal., {\bf 6} (1970), 172 -- 191.

\bibitem{Vasilescu} F. H. Vasilescu, 
                    \textit{Analytic Functional Calculus and Spectral Decompositions,}
                     Editura Academiei: Bucuresti, Romania and D. Reidel Publishing Company, (1982).

%












\end{thebibliography}
\end{document}